\documentclass[11pt]{article}
\usepackage{amssymb, latexsym}
\newtheorem{defin}{}
\newtheorem{saetze}[defin]{}
\newtheorem{conjec}[defin]{}
\newtheorem{lemmas}[defin]{}
\newtheorem{folger}[defin]{}
\newtheorem{bemerk}[defin]{}

\newenvironment{theorem}  {\begin{saetze}\it {\bf Theorem:}}{\end{saetze}}

\newenvironment{remark}   {\begin{bemerk}\it {\bf Remark:}}{\end{bemerk}}
\newenvironment{proof}    {{\it Proof}:}{{\hfill \fillbox \bigskip}}

\newcommand{\fillbox}{\mbox{$\bullet$}}
\newcommand{\ra}{\rightarrow}

\newcommand{\ol}{\overline}

\newcommand{\split}{\rtimes}
\newcommand{\N}{\mathbb N}
\newcommand{\F}{\mathbb F}

\newcommand{\GL}{\mathrm{GL}}
\newcommand{\Aut}{\mathrm{Aut}}

\newcommand{\ind}{\mathrm{ind}}
\newcommand{\DC}{\mathrm{DC}}
\newcommand{\GN}{\mathcal N}
\newcommand{\w}{w}
\renewcommand{\O}{\mathcal O}

\renewcommand{\SS}{\mathcal S}
\newcommand{\K}{\mathcal K}

\newenvironment{items}{\begin{list}{$\alph{item})$}
{\labelwidth18pt \leftmargin18pt \topsep3pt \itemsep1pt \parsep0pt}}
{\end{list}}

\begin{document}

\title{Enumeration of groups \\ whose order factorises in at most $4$ primes}
\author{Bettina Eick}
\date{\today}
\maketitle

\begin{abstract}
Let $\GN(n)$ denote the number of isomorphism types of groups of order $n$. 
We consider the integers $n$ that are products of at most $4$ not necessarily
distinct primes and exhibit formulas for $\GN(n)$ for such $n$.
\end{abstract}

\section{Introduction}

The construction up to isomorphism of all groups of a given order $n$ is 
an old and fundamental problem in group theory. It has been initiated by 
Cayley \cite{Cay54} who determined the groups of order at most $6$. Many
publications have followed Cayley's work; A history of the problem can
be found in \cite{BEO02}. 

The enumeration of the isomorphism types of groups of order $n$ is a related
problem. The number $\GN(n)$ of isomorphism types of groups of order $n$ is 
known 
for all $n$ at most $2~000$, see \cite{BEO02}, and for almost all $n$ at 
most $20~000$, see \cite{EHH16}. Asymptotic estimates for $\GN(n)$ have
been determined in \cite{Pyb98}. However, there is no closed formula known 
for $\GN(n)$ as a function in $n$. Many details on the group enumeration 
problem can be found in \cite{CDOB08} and \cite{BNV07}. 

Higman \cite{Hig60b} considered prime-powers $p^m$. His PORC conjecture 
suggests that $\GN(p^m)$ as a function in $p$ is PORC (polynomial on 
residue classes). This has been proved for all $m \leq 7$, see H\"older 
\cite{Hol93} for $m \leq 4$, Bagnera \cite{Bag99} or Girnat \cite{Gir03} 
for $m = 5$, Newman, O'Brien \& Vaughan-Lee \cite{NOV04} for $m=6$ and 
O'Brien \& Vaughan-Lee \cite{OVL05} for $m=7$. To exhibit the flavour of 
the results, we recall the explicit PORC polynomials for $\GN(p^m)$ for 
$m \leq 5$ as follows.

\begin{theorem} {\bf (H\"older \cite{Hol93}, Bagnera \cite{Bag99})} 
\begin{items}
\item[$\bullet$]
$\GN(p^1) = 1$ for all primes $p$.
\item[$\bullet$]
$\GN(p^2) = 2$ for all primes $p$.
\item[$\bullet$]
$\GN(p^3) = 5$ for all primes $p$.
\item[$\bullet$]
$\GN(2^4) = 14$ and $\GN(p^4) = 15$ for all primes $p \geq 3$.
\item[$\bullet$]
$\GN(2^5) = 51$, $\GN(3^5) = 71$ and 
$\GN(p^5) = 2p + 61 + 2 \gcd(p-1,3) + \gcd(p-1,4)$ for all primes
$p \geq 5$.
\end{items}
\end{theorem}

H\"older \cite{Hol95b} determined a formula for $\GN(n)$ for all square-free
$n$. For $m \in \N$ let $\pi(m)$ denote the number of different primes dividing 
$m$. For $m \in \N$ and a prime $p$ let $c_m(p)$ denote the number of prime 
divisors $q$ of $m$ with $q \equiv 1 \bmod p$. The following is also proved
in \cite[Prop. 21.5]{BNV07}.

\begin{theorem} {\bf (H\"older \cite{Hol95b})} \\
Let $n \in \N$ be square free. Then 
\[ \GN(n) = \sum_{m \mid n} \; \prod_{p \in \pi(\frac{n}{m})} 
    \frac{p^{c_m(p)}-1}{p-1}.\]
\end{theorem}

The aim here is to determine formulas for $\GN(n)$ if $n$ is a product of 
at most $4$ primes. If $n$ is a prime-power or is square-free, then such 
formulas follow from the results cited above. Hence it remains to consider 
the numbers $n$ that factorise as $p^2 q$, $p^3 q$, $p^2 q^2$ or $p^2 q r$ 
for different primes $p,q$ and $r$. For each of these cases we determine an 
explicit formula for $\GN(n)$, see Theorems \ref{thp2q}, \ref{thp3q}, 
\ref{thp2q2} and \ref{thp2qr}. Each of these formulas is a polynomial on 
residue classes; that is, there are finitely many sets of number-theoretic 
conditions on the involved primes so that $\GN(n)$ is a polynomial in the 
involved primes for each of the condition-sets. We summarise this in the
following theorem.

\begin{theorem} \label{PORC}
{\bf (See Theorems \ref{thp2q}, \ref{thp3q}, \ref{thp2q2}, \ref{thp2qr}
below)} \\
Let $p,q$ and $r$ be different primes and $n \in \{ p^2q, p^3q, p^2q^2,
p^2qr\}$. Then $\GN(n)$ is a polynomial on residue classes.
\end{theorem}

The enumerations obtained in this paper overlap with various known 
results. For example, H\"older \cite{Hol93} 
considered the groups of order $p^2 q$, Western \cite{Wes99} those of 
order $p^3 q$, Le Vavasseur \cite{Vav99, Vav02} and Lin \cite{Lin74} 
those of order $p^2 q^2$ and Glenn \cite{Gle06} those of order $p^2 qr$.
Moreover, Laue \cite{Lau82} considered all orders of the form $p^a q^b$
with $a+b \leq 6$ and $a < 5$ and $b < 5$ as well as the orders dividing
$p^2 q^2 r^2$. 

So why are these notes written? There are two reasons. First, they
provide a uniform and reasonably compact proof for the considered group
numbers and they exhibit the resulting group numbers as a closed formula
with few case distinctions. Laue's work \cite{Lau82} also contains a 
unified approach towards the determination of its considered groups and
this approach is similar to ours, but it is not easy to read and to 
extract the results. Our second aim with these notes is to provide a 
uniform and reliable source for the considered group enumerations. The 
reliability of our results is based on its proofs as well as on a
detailed comparision with the Small Groups library \cite{smallgroups}.

\section*{Acknowledgements}

We give more details on the results available in the literature that 
overlap with the results here in the discussions before the Theorems 
\ref{thp2q}, \ref{thp3q}, \ref{thp2q2} and \ref{thp2qr}. Most of these
details have been provided by Mike Newman. The author thanks Mike Newman 
for this and also for various discussions on these notes.

\section{Divisibility}

For $r, s \in \N$ we define the function $\w_r(s)$ via $\w_r(s) = 1$ if 
$s \mid r$ and $\w_r(s) = 0$ otherwise. The following remark exhibits
the relation of $\w_r(s)$ and the underlying $\gcd$'s.

\begin{remark}
For $r, s \in \N$ it follows that
\[ \w_r(s) = \prod_{d \mid s, d \neq s} \frac{\gcd(r,s) - d}{s - d}. \]
\end{remark}

\section{Counting subgroups of linear groups}

For $r \in \N$ and a group $G$ we denote with $s_r(G)$ the number of
conjugacy classes of subgroups of order $r$ in $G$. We recall the 
following well-known result.

\begin{remark} \label{irred}
Let $n \in \N$ and $p$ prime. Then $\GL(n,p)$ has an irreducible 
cyclic subgroup of order $m$ if and only if $m \mid (p^n - 1)$ and 
$m \nmid (p^d - 1)$ for each $d < n$. Further, if there exists an 
irreducible cyclic subgroup of order $m$ in $\GL(n,p)$, then it is 
unique up to conjugacy. 
\end{remark}

The next theorem counts the conjugacy classes of subgroups of certain
orders in $\GL(2,p)$. With $C_r^s$ we denote the $s$-fold direct product
of cyclic groups of order $r$.

\begin{theorem}
\label{linear2}
Let $p,q$ and $r$ be different primes and let $G = \GL(2,p)$.
\begin{items}
\item[\rm (a)]
$s_2(G) = 2$ and for $q > 2$ it follows that
\[ s_q(G) = \frac{q+3}{2} \w_{p-1}(q) + \w_{p+1}(q). \] 
\item[\rm (b)]
$s_4(G) = 2 + 3 \w_{p-1}(4)$ and for $q > 2$ it follows that
\[ s_{q^2}(G) = \w_{p-1}(q) + \frac{q^2+q+2}{2} \w_{p-1}(q^2) + \w_{p+1}(q^2).\]
\item[\rm (c)] 
$s_{2r}(G) = \frac{3r+7}{2} \w_{p-1}(r) + 2 \w_{p+1}(r)$ and 
for $r > q > 2$ it follows that
\[ s_{qr}(G) = \frac{qr + q + r + 5}{2} \w_{p-1}(qr) + 
      \w_{p^2-1}(qr)(1-\w_{p-1}(qr)).\]
\end{items}
\end{theorem}

\begin{proof}
Let $m \in \N$ with $p \nmid m$. As a preliminary step in this proof, we 
investigate the number of conjugacy classes of cyclic subgroups of order
$m$ in $\GL(2,p)$. By Remark \ref{irred}, an irreducible subgroup of this 
form exists if and only if $m \mid (p^2-1)$ and $m \nmid (p-1)$ and its 
conjugacy class is unique in this case. A reducible cyclic subgroup of 
order $m$ in $\GL(2,p)$ embeds into the group of diagonal matrices $D$. Note 
that $D \cong C_{p-1}^2$ and thus $\GL(2,p)$ has a reducible cyclic subgroup 
of order $m$ if and only if $m$ divides the exponent $p-1$ of $D$. If 
$m \mid (p-1)$, then there exists a unique subgroup $U \cong C_m^2$ in 
$D$. This subgroup $U$ contains every cyclic subgroup of order $m$ in $D$. 
The group $\GL(2,p)$ acts on $D$ and on $U$ by permutation of the diagonal 
entries of an element of $D$. 

(a) Each group of prime order $q$ is cyclic. An irreducible subgroup of order
$q$ exists if and only if $q \neq 2$ and $q \mid (p+1)$, since $p^2-1 = (p-1)
(p+1)$ and $\gcd(p-1,p+1) \mid 2$. A reducible subgroup of order $q$ exists 
if and only $q \mid (p-1)$. In this case, the number of conjugacy classes of 
reducible cyclic subgroups of order $q$ in $\GL(2,p)$ can be enumerated as 
$2$ if $q = 2$ and $(q+3)/2$ otherwise, as there are $(q+1)$ subgroups of 
order $q$ in $C_q^2$ and all but the subgroups with diagonals of the form 
$(a,a)$ or $(a,a^{-1})$ for $a \in \F_p^*$ have orbits of length two under 
the action of $\GL(2,p)$ by permutation of diagonal entries. 

(b) We first consider the cyclic subgroups of order $q^2$ in $\GL(2,p)$.  
For the irreducible case, note that if $q^2 \mid (p^2-1)$ and $q^2 \nmid 
(p-1)$, then either $q = 2$ and $4 \nmid (p-1)$ or $q > 2$ and $q^2 \mid 
(p+1)$. For the reducible case we note that if $q^2 \mid (p-1)$, then there 
are $(q^2 + q + 2)/2$ conjugacy classes of reducible cyclic subgroups
of order $q^2$ in $\GL(2,p)$, as there are $(q^2+q)$ subgroups and all 
but those with diagonal of the form $(a,a)$ and $(a,a^{-1})$ for $a \in 
\F_p^*$ have orbits of length two. Thus the number of  conjugacy classes 
of cyclic subgroups of order 
$q^2$ in $\GL(2,p)$ is $1 + 3 \w_{p-1}(4)$ if $q = 2$ and $(q^2 + q + 2)/2 
\cdot \w_{p-1}(q^2) + \w_{p+1}(q^2)$ if $q > 2$. It remains to consider 
the subgroups of type $C_q^2$. Such a subgroup is reducible and exist if 
$q \mid (p-1)$. In this case there exists a unique conjugacy class of such 
subgroups.
 
(c)
We first consider the cyclic subgroups of order $qr$ in $\GL(2,p)$. If $q = 2$, 
then $p \neq 2$. Thus $2r \mid (p^2-1)$ and $2r \nmid (p-1)$ if and only if 
$r \mid (p+1)$. As in the previous cases, this yields that there are 
$(3r + 5)/2 \cdot \w_{p-1}(r) + \w_{p+1}(r)$ cyclic subgroups of order 
$2r$ in $\GL(2,p)$. If $q > 2$, then $r > 2$.  The number of cyclic subgroups
of order $qr$ in $\GL(2,r)$ in this case is $(qr+r+q+5)/2 \cdot \w_{p-1}(qr) 
+ \w_{p^2-1}(qr)(1-\w_{p-1}(qr)$ using the same arguments as above. It remains 
to consider the case of non-cyclic subgroups. Such a subgroup $H$ is 
irreducible and satisfies $q = 2$. If $H$ is imprimitive, then $C_r$
is diagonalisable; there is one such possibility if $r \mid (p-1)$. If
$H$ is primitive, then $C_r$ is irreducible; there is one such possibility
if $r \mid (p+1)$. 
\end{proof}

We extend Theorem \ref{linear2} with the following.

\begin{remark} \label{norms}
Let $p$ and $q$ be different primes and let $G = \GL(2,p)$. If $H$ is a 
subgroup of order $q$ in $G$, then $[N_G(H):C_G(H)] \mid 2$.
The number of groups $H$ of order $q$ in $G$ satisfying 
$[N_G(H) : C_G(H)] = 2$ is $0$ for $q = 2$ and $w_{p-1}(q) + w_{p+1}(q)$ 
for $q > 2$.
\end{remark}

\begin{proof}
Consider the groups $H$ of order $q$ in $\GL(2,p)$. If $H$ is irreducible,
then $H$ is a subgroup of a Singer cycle and this satisfies $[N_G(H):C_G(H)]
= 2$. If $H$ is reducible, then $H$ is a subgroup of the group of diagonal
matrices $D$. The group $D$ satisfies $[N_G(D):C_G(D)] = 2$, where the 
normalizer acts by permutation of the diagonal entries. As in the proof
of Theorem \ref{linear2} (a), only the group $H$ with diagonal of the form
$(a,a)$ for $a \in \F_p^*$ satisfies $[N_G(H):C_G(H)] = 2$.
\end{proof}

\begin{theorem}
\label{linear3}
Let $p$ be a prime, let $G = \GL(3,p)$ and let $q \neq p$.
Then $s_2(G) = 3$ and for $q > 2$ it follows that
\[ s_q(G) = \frac{q^2 + 4q + 9 + 4 \w_{q-1}(3)}{6} \w_{p-1}(q) 
              + \w_{(p+1)(p^2+p+1)}(q)(1 - \w_{p-1}(q)). \] 
\end{theorem}

\begin{proof}
We first consider the diagonalisable subgroups of order $q$ in $\GL(3,p)$.
These exist if $q \mid (p-1)$. If this is the case, then the group $D$ of
diagonal matrices has a subgroup of the form $C_q^3$ and this contains all
subgroups of order $q$ in $D$. The group $D$ has $q^2 + q + 1$ subgroups 
of order $q$ and these fall under the permutation action of diagonal entries
into $3$ orbits if $q = 2$, into $(q^2+4q+9)/6$ orbits if $q > 2$ and $3 
\nmid (q-1)$ and into $(q^2+4q+13)/6$ orbits if $q > 2$ and $3 \mid (q-1)$. 
Next, we consider the groups that are not diagonalisable. These 
can arise from irreducible subgroups in $\GL(2,p)$ or in $\GL(3,p)$. In 
the first case, there exists one such class if $q > 2$ and $q \mid (p+1)$
as in Theorem \ref{linear2}.
In the second case, by Remark \ref{irred} there exists one such class if 
$q \mid (p^3-1)$ and $q \nmid (p^2-1)$ and $q \nmid (p-1)$. 
Note that the two cases are mutually exclusive. In summary, there exists 
an irreducible subgroup
of order $q$ in $\GL(3,p)$ if $q > 2$ and $q \mid (p+1)(p^2+p+1)$ and 
$q \nmid (p-1)$.
\end{proof}

We note that $\gcd( (p+1)(p^2+p+1), p-1) \mid 6$ for each prime $p$. 
Thus for $q \geq 5$ it follows that $\w_{(p+1)(p^2+p+1)}(q) 
= \w_{(p+1)(p^2+p+1)}(q) (1-\w_{p-1}(q))$ which simplifies the formula
in Theorem \ref{linear3}.

\section{Counting split extensions}

For two groups $N$ and $U$ let $\Pi(U,N)$ denote the set of all group 
homomorphisms $\varphi : U \ra \Aut(N)$. The direct product $\Aut(U) 
\times \Aut(N)$ acts on the set $\Pi(U,N)$ via 
\[ \varphi^{(\alpha, \beta)}(g) = \beta^{-1} (\varphi(\alpha^{-1}(g))) \beta\]
for $(\alpha, \beta) \in \Aut(U) \times \Aut(N)$, $\varphi \in \Pi(U,N)$
and $g \in N$. If $\ol{\beta}$ is the conjugation by $\beta$ in $\Aut(N)$, then
this action can be written in short form as
\[ \varphi^{(\alpha, \beta)} = \ol{\beta} \circ \varphi \circ \alpha^{-1}. \]
Given $\varphi \in \Pi(U, N)$, the stabilizer of $\varphi$ in $\Aut(U) 
\times \Aut(N)$ is called the group of compatible pairs and is denoted by 
$Comp(\varphi)$. If $N$ is abelian, then $N$ is an $U$-module via $\varphi$
for each $\varphi \in \Pi(U,N)$. In this case $Comp(\varphi)$ acts on 
$H^2_\varphi(U, N)$ induced by its action on $Z^2_\varphi(U,N)$ via
\[ \gamma^{(\alpha, \beta)}(g,h) 
   = \beta^{-1}(\gamma(\alpha^{-1}(g), \alpha^{-1}(h)))\]
for $\gamma \in Z^2_\varphi(U, N)$, $(\alpha, \beta) \in Comp(\varphi)$ 
and $g,h \in U$.  These constructions can be used to solve the isomorphism 
problem for extensions in two different settings. We recall this in the 
following.

\subsection{Extensions with abelian kernel}

Suppose that $N$ is abelian and that $N$ is fully invariant in each extension
of $N$ by $U$; this is, for example, the case if $N$ and $U$ are coprime or
if $N$ maps onto the Fitting subgroup in each extension of $N$ by $U$.
The following theorem seems to be folklore.

\begin{theorem} 
\label{Compat}
Let $N$ be finite abelian and $U$ be a finite group so that $N$ is fully 
invariant in each extension of $N$ by $U$. Let $\O$ be a complete set 
of representatives of the $\Aut(U) \times \Aut(N)$ orbits in $\Pi(U,N)$
and for each $\varphi \in \O$ let $o_\varphi$ denote the number of orbits 
of $Comp(\varphi)$ on $H^2_\varphi(U,N)$. Then the number of isomorphism
types of extensions of $N$ by $U$ is 
\[ \sum_{\varphi \in \O} o_\varphi.\]
\end{theorem}

The following theorem proved in \cite[Th. 14]{DEi05} exploits the situation 
further in a special case. Again, $C_l$ denotes the cyclic group of order $l$.

\begin{theorem} {\bf (Dietrich \& Eick \cite{DEi05})} \\
\label{cubefree}
Let $p$ be a prime, let $N \cong C_p$, and let $U$ be finite with Sylow
$p$-subgroup $P$ so that $P \cong C_p$. Then there are either one or two 
isomorphism types of extensions of $N$ by $U$. There are two isomorphism 
types of extensions if and only if $N$ and $P$ are isomorphic as 
$N_U(P)$-modules.
\end{theorem}

\subsection{A special type of split extensions}

In this section we recall a variation of a theorem by Taunt \cite{Tau54}. 
As a preliminary step we introduce some notation. Let $N$ and $U$ be 
finite solvable groups of coprime order. Let $\SS$ denote a set 
of representatives for the conjugacy classes of subgroups in $\Aut(N)$,
let $\K$ denote the set of representatives for the $\Aut(U)$-classes of 
normal subgroups in $U$ and let $\O = \{ (S, K) \mid S \in \SS, K \in \K 
\mbox{ with } S \cong U/K \}$. For $(S,K) \in \O$ let $\iota : U/K \ra
S$ denote a fixed isomorphism, let $A_K$ denote the subgroup of $\Aut(U/K)$ 
induced by the action of $Stab_{\Aut(U)}(K)$ on $U/K$, and denote with 
$A_S$ the subgroup of $\Aut(U/K)$ induced by the action of $N_{\Aut(N)}(S)$ 
on $S$ and thus, via $\iota$, on $U/K$. Then the double cosets of the 
subgroups $A_K$ and $A_S$ in $\Aut(U/K)$ are denoted by 
\[ \DC(S, K) := A_K \setminus \Aut(U/K) \; / \; A_S.\]

\begin{theorem} \label{Taunt}
Let $N$ and $U$ be finite solvable groups of coprime order. Then the number 
of isomorphism types of split extensions $N \split U$ is 
\[ \sum_{(S,K) \in \O} |\DC(S,K)|.\]
\end{theorem}

\begin{proof}
Taunt's theorem \cite{Tau54} claims that the number of isomorphism types of
split extensions $N \split U$ correspond to the orbits of $\Aut(U) \times
\Aut(N)$ on $\Pi(U,N)$. In turn, these orbits correspond to the union of 
orbits of $A_K \times N_{\Aut(N)}(S)$ on the set of isomorphisms $U/K \ra 
S$. The latter translate to the double cosets $\DC(S,K)$.
\end{proof}

We apply Theorem \ref{Taunt} in two special cases in the following. 
Again, let $C_l$ denote the cyclic group of order $l$.

\begin{theorem} \label{Taunt1}
Let $q^k$ be a prime-power, let $N = C_{q^k}$ and let $U$ be a finite group 
of order coprime to $q$. Denote $\pi = \gcd(|U|, q^{k-1}(q-1))$.
For $l \mid \pi$, let $\K_l$ be a set of representatives of the 
$Aut(U)$-classes of normal subgroups $K$ in $U$ with $U/K \cong C_l$. 
For $K \in \K_l$ let $\ind_K = [\Aut(U/K) : A_K]$. If $k \leq 2$ or 
$q$ is odd, then the number of isomorphism types of split extensions 
$N \split U$ is 
\[ \sum_{l \mid \pi} \sum_{K \in \K_l}  \ind_K. \]
\end{theorem}

\begin{proof}
We apply Theorem \ref{Taunt}.
The group $\Aut(N)$ is cyclic of order
$p^{k-1}(p-1)$. Hence for each $l \mid \pi$ there exists a unique 
subgroup $S$ of order $l$ in $\Aut(N)$ and this subgroup is cyclic. 
Next, as $\Aut(N)$ is abelian, it follows that $N_{\Aut(N)}(S) 
= \Aut(N)$ and $\Aut(N)$ acts trivially on $S$ by conjugation. Hence 
$A_S$ is the trivial group and $|\DC(S,K)| = \ind_K$ for each $K$.
\end{proof}

\begin{theorem} \label{Taunt2}
Let $q^k$ be a prime-power, let $U = C_{q^k}$ and let $N$ be a finite group 
of order coprime to $q$. Let $\SS$ be a set of conjugacy class representatives
of cyclic subgroups of order dividing $q^k$ in $\Aut(N)$. Then the number
of isomorphism types of split extensions $N \split U$ equals $|\SS|$.
\end{theorem}

\begin{proof}
Again we use Theorem \ref{Taunt}.
For each divisor $p^l$ of $p^k$ there exists a unique normal subgroup $K$ 
in $U$ with $|K| = p^l$ and $U/K$ is cyclic of order $p^{k-l}$. Note that 
$A_K = \Aut(U/K)$ for each such $K$. Hence $|\DC(S,K)| = 1$ in all cases.
\end{proof}

\section{Groups of order $p^2q$}

The groups of order $p^2 q$ have been considered by H\"older 
\cite{Hol93}, by Lin \cite{Lin74}, by Laue \cite{Lau82} and in various
other places. The results by H\"older, Lin and Laue agree with ours. 
(Lin's results have some harmless typos). We also refer to 
\cite[Prop. 21.17]{BNV07} for an alternative description and proof of 
the following result.

\begin{theorem} \label{thp2q} 
Let $p$ and $q$ be different primes.
\begin{items}
\item[\rm (a)]
If $q = 2$, then $\GN(p^2 q) = 5$.
\item[\rm (b)]
If $q > 2$, then 
\[ \GN(p^2 q) = 2 + (q+5)/2 \cdot \w_{p-1}(q) + \w_{p+1}(q) 
       + 2 \w_{q-1}(p) + \w_{q-1}(p^2).\]
\end{items}
\end{theorem}

\begin{proof}
The classification of groups of order $p^2$ yields that there are two 
nilpotent groups of order $p^2 q$: the groups $C_{p^2} \times C_q$ and 
$C_p^2 \times C_q$. It remains to consider the non-nilpotent groups of 
the desired order. Note that
every group of order $p^2 q$ is solvable by Burnside's theorem.
\medskip

{\em Non-nilpotent groups with normal $p$-Sylow subgroup.} 
These groups have the form 
$N \split U$ with $|N| = p^2$ and $U = C_q$. We use Theorem \ref{Taunt2} 
to count the number of such split extensions. Thus we count the number 
of conjugacy classes of subgroups of order $q$ in $Aut(N)$. 
\begin{items}
\item[$\bullet$]
$N \cong C_p^2$. Then $\Aut(N) \cong GL(2,p)$ and the number of conjugacy 
classes of subgroups of order $q$ in $\Aut(N)$ is exhibited in Theorem 
\ref{linear2} (a).
\item[$\bullet$]
$N \cong C_{p^2}$. Then $\Aut(N) \cong C_{p (p-1)}$ is cyclic. Thus there 
is at most one subgroup of order $q$ in $Aut(N)$ and this exists if and 
only if $q \mid (p-1)$.
\end{items}
\medskip

{\em Non-nilpotent groups with normal $q$-Sylow subgroup.} 
These groups have the form 
$N \split U$ with $N = C_q$ and $|U| = p^2$. We use Theorem \ref{Taunt1} 
to count the number of such split extensions. For this purpose we have 
to consider the $\Aut(U)$-classes of proper normal subgroups $K$ in $U$ 
with $U/K$ cyclic.
\begin{items}
\item[$\bullet$]
$U \cong C_{p^2}$. Then $U$ has a two proper normal subgroups $K$ with 
cyclic quotient. The case $K = 1$ arises if and only if $p^2 \mid (q-1)$ 
and the case $K = C_p$ arises if and only if $p \mid (q-1)$. In both cases 
$\ind_K = 1$.
\item[$\bullet$]
$U \cong C_p^2$. Then there exists one $\Aut(U)$-class of normal subgroups 
$K$ with cyclic quotient in $U$ and this has the form $K = C_p$ and yields
$\ind_K = 1$. This case arises if $p \mid (q-1)$.
\end{items}
\medskip

{\em Groups without normal Sylow subgroup.} Let $G$ be such a group and 
let $F$ be the Fitting subgroup of $G$. As $G$ is solvable and non-nilpotent, 
it follows that $1 < F < G$. As $G$ has no normal Sylow subgroup, we obtain
that $pq \mid [G : F]$. Thus $|F| = p$ and $|G/F| = pq$ is the only option.
Next, $G/F$ acts on faithfully on $F$ by conjugation, since $F$ is the 
Fitting subgroup. Hence $pq \mid |Aut(F)| = (p-1)$ and this is a contradiction. 
Thus this case cannot arise.
\end{proof}

\section{Groups of order $p^3q$}

The groups of order $p^3 q$ have been determined by Western \cite{Wes99} 
and Laue \cite{Lau82}. Western's paper is essentially correct,
but the final summary table of groups has a group missing in the case that
$q \equiv 1 \bmod p$; the missing group appears in Western's analysis in 
Section 13. There are further minor issues in Section 32 of Western's paper.
There are disagreements between our results and the results of Laue
\cite[p. 224-6]{Lau82} for the case $p=2$ and the case $q=3$. We have not 
tried to track the origin of these in Laue's work. 

\begin{theorem} \label{thp3q}
Let $p$ and $q$ be different primes.
\begin{items}
\item[\rm (a)]
There are two special cases
$\GN(3 \cdot 2^3) = 15$ and $\GN(7 \cdot 2^3) = 13$.
\item[\rm (b)]
If $q = 2$, then $\GN(p^3 q) = 15$ for all $p > 2$.
\item[\rm (c)]
If $p = 2$, then $\GN(p^3 q) = 12 + 2 \w_{q-1}(4) + \w_{q-1}(8)$ for all 
$q \neq 3,7$.
\item[\rm (d)]
If $p$ and $q$ are both odd, then 
\begin{eqnarray*}
\GN(p^3 q) &=& 5 + (q^2 + 13 q + 36)/6 \cdot \w_{p-1}(q) \\
           && + (p+5) \cdot \w_{q-1}(p) \\
           && + 2/3 \cdot \w_{q-1}(3) \w_{p-1}(q)  
              + \w_{(p+1)(p^2+p+1)}(q)(1- \w_{p-1}(q)) \\
           && + \w_{p+1}(q) + 2 \w_{q-1}(p^2) + \w_{q-1}(p^3).
\end{eqnarray*}
\end{items}
\end{theorem}

Before we embark on the proof of Theorem \ref{thp3q}, we note that the 
formula of Theorem \ref{thp3q} d) can be simplified by 
distinguishing the cases $q=3$ and $q \neq 3$. For $q = 3$ and $p$ odd it 
follows that $\w_{(p+1)(p^2+p+1)}(q)(1- \w_{p-1}(q)) = \w_{p+1}(3)$. Thus 
Theorem \ref{thp3q} d) for $q = 3$ and $p$ odd reads
$\GN(3 p^3) = 5 + 14 \w_{p-1}(3) + 2 \w_{p+1}(3)$.
If $q > 3$ and $p$ is odd, then 
$\w_{(p+1)(p^2+p+1)}(q)(1- \w_{p-1}(q)) = \w_{(p+1)(p^2+p+1)}(q)$ holds. 
Again, this can be used to simplify the formula of Theorem \ref{thp3q} d). 
\bigskip

\begin{proof}
The proof follows the same strategy as the proof of Theorem \ref{thp2q}.
Burnside's theorem asserts that every group of order $p^3 q$ is solvable.
It is easy to see that there are five nilpotent groups of order $p^3 q$:
the groups $G \times C_q$ with $|G| = p^3$. It remains to consider the 
non-nilpotent groups of the desired order.
\medskip

{\em Non-nilpotent groups with normal $p$-Sylow subgroup.} These groups have 
the form $N \split U$ with $|N| = p^3$ and $U = C_q$. Using Theorem 
\ref{Taunt2}, they correspond to the conjugacy classes of subgroups of order
$q$ in $\Aut(N)$. There are five isomorphism types of groups $N$ of order
$p^3$. For $p=2$, the groups $\Aut(Q_8)$ and $\Aut(C_2^3) = GL(3,2)$ 
have subgroups of order coprime to $2$; that is, $\Aut(Q_8)$ has one
conjugacy class of subgroups of order $3$ and $GL(3,2)$ has one conjugacy
class of subgroups of order $3$ and $7$. This leads to the special cases 
in a) and shows that in all other cases on $q$ this type of group does not 
exist for $p=2$. It remains to consider the case $p > 2$.
\begin{items}
\item[$\bullet$]
$N \cong C_{p^3}$: Then $\Aut(N)$ is cyclic of order $p^2(p-1)$. 
Thus $\Aut(N)$ has at most one subgroup of order $q$ and this exists if
and only if $q \mid (p-1)$. This adds 
\[ c_1 := \w_{p-1}(q).\]
\item[$\bullet$]
$N \cong C_{p^2} \times C_p$: In this case $\Aut(N)$ is solvable and
has a normal Sylow $p$-subgroup with $p$-complement of the form $C_{p-1}^2$.
Thus $\Aut(N)$ contains a subgroup of order $q$ if and only if $q \mid (p-1)$. 
In this case there are $q+1$ such subgroups in $C_q^2$ and these translate
to conjugacy classes of such subgroups in $\Aut(N)$. This adds
\[ c_2 := (q+1) \w_{p-1}(q).\]
\item[$\bullet$]
$N$ is extraspecial of exponent $p$: Then $Aut(N) \ra Aut(N/\phi(N))
\cong GL(2,p)$ is surjective. Thus the conjugacy classes of subgroups of
order $q$ in $\Aut(N)$ correspond to the conjugacy classes of subgroups of
order $q$ in $GL(2,p)$. These are counted in Theorem \ref{linear2} (a). Hence
this adds
\begin{eqnarray*}
 c_3 &:=& \frac{q+3}{2} \w_{p-1}(q) + \w_{p+1}(q) \mbox{ if } q > 2, \\
 c_3 &:=& 2 \mbox{ if } q = 2.
\end{eqnarray*}
\item[$\bullet$]
$N$ is extraspecial of exponent $p^2$: Then $\Aut(N)$ is solvable and
has a normal Sylow $p$-subgroup with $p$-complement of the form $C_{p-1}$.
Thus there is at most one subgroup of order $q$ in $\Aut(N)$ and this exists 
if and only if $q \mid (p-1)$. This adds 
\[ c_4 := \w_{p-1}(q).\]
\item[$\bullet$]
$N \cong C_p^3$: Then $\Aut(N) \cong \GL(3,p)$. The conjugacy classes of 
subgroups of order $q$ in $\GL(3,p)$ are counted in Theorem \ref{linear3}.
Hence this adds
\begin{eqnarray*}
 c_5 &:=& \frac{1}{6}(q^2+4q+9 + 4\w_{q-1}(3)) \w_{p-1}(q)
          + \w_{(p+1)(p^2+p+1)}(q)(1-\w_{p-1}(q)) \mbox{ if } q > 2,\\
 c_5 &:=& 3 \mbox{ if } q = 2.
\end{eqnarray*}
\end{items}
\medskip

{\em Non-nilpotent groups with normal $q$-Sylow subgroup.} These groups 
have the form $N \split U$ with $N = C_q$ and $|U| = p^3$. Using Theorem
\ref{Taunt1}, we have to determine the $\Aut(U)$-orbits of proper normal 
subgroups $K$ in $U$ with $U/K$ cyclic of order dividing $q-1$ and then 
for each such $K$ determine $\ind_K$.
\begin{items}
\item[$\bullet$]
$U \cong C_{p^3}$: In this case there are the options $K \in \{1, C_p,
C_{p^2}\}$ and all of these have $\ind_K = 1$. Hence this adds
\[ c_6 := \w_{q-1}(p) + \w_{q-1}(p^2) + \w_{q-1}(p^3).\]
\item[$\bullet$]
$U \cong C_{p^2} \times C_p$: In this case there are two normal subgroups
$K$ with $U/K \cong C_p$ and there is one normal subgroup $K$ with $U/K
\cong C_{p^2}$ up to $\Aut(U)$-orbits. All of these have $\ind_K = 1$. Thus
this adds 
\[ c_7 := 2 \w_{q-1}(p) + \w_{q-1}(p^2).\]
\item[$\bullet$]
$U$ extraspecial of exponent $p$ (or $Q_8$): In this case there is one 
$\Aut(U)$-orbit of normal subgroups $K$ with $U/K \cong C_p$ and this has 
$\ind_K = 1$. Hence this adds 
\[ c_8 := \w_{q-1}(p).\]
\item[$\bullet$]
$U$ extraspecial of exponent $p^2$ (or $D_8$): Then there are two 
$\Aut(U)$-orbits of normal subgroups $K$ with $U/K \cong C_p$; one has 
$\ind_K = 1$ and the other has $\ind_K = p-1$. Thus this case adds 
\[ c_9 := p \w_{q-1}(p).\]
\item[$\bullet$]
$U \cong C_p^3$: In this case there is one $\Aut(U)$-orbit of normal
subgroups $K$ with $U/K \cong C_p$ and this has $\ind_k = 1$. Thus this
adds 
\[ c_{10} := \w_{q-1}(p).\]
\end{items}
\medskip

{\em Groups without normal Sylow subgroup.} Let $G$ be such a group and
let $F$ be the Fittingsubgroup of $G$. Since $G$ is solvable, it follows
that $F$ is not trivial. Further, $pq \mid [G : F]$ by construction. Thus 
$|F| = p$ or $|F| = p^2$ are the only options. Recall that $G/F$ acts 
faithfully on $F/\phi(F)$. If $F$ is cyclic of order $p$ or $p^2$, then 
a group of order $pq$ cannot act faithfully on $F/\phi(F)$. Hence $F
\cong C_p^2$ is the only remaining possibility. 
In this case $G/F$ has order $pq$ and embeds into $Aut(F) \cong GL(2,p)$;
thus $q \mid (p+1)(p-1)$. As $F$ is the Fitting subgroup of $G$ of order
$p^2$, it follows that $G/F$ cannot have a normal subgroup of order $p$.
Hence $G/F \cong C_q \split C_p$ and $p \mid (q-1)$. This is only possible
for $p=2$ and $q=3$ and thus is covered by the special case for groups of
order $3 \cdot 2^3$.
\medskip

Except for the special cases, it now remains to sum up the values $5 
+ \sum_{i=6}^{10} c_i$ for $p = 2$ and $5 + \sum_{i=1}^{10} c_i$ for
$p > 2$. This yields the formulas exhibited in the theorem.
\end{proof}

\section{Groups of order $p^2q^2$}

The groups of order $p^2 q^2$ have been determined by Lin \cite{Lin74}, 
Le Vavasseur \cite{Vav02} and Laue \cite{Lau82}. Lin's work is essentially 
correct (it only has minor mistakes) and it agrees with the work by Laue 
\cite[p. 214-43]{Lau82} and our results. Lin seems unaware of the work by 
Le Vavasseur \cite{Vav02}. We have not compared our results with those of
Le Vavasseur.

\begin{theorem} \label{thp2q2}
Let $p$ and $q$ be different primes with $p < q$.
\begin{items}
\item[\rm (a)]
There is one special case $\GN(2^2 \cdot 3^2) = 14$.
\item[\rm (b)]
If $p = 2$, then $\GN(p^2 q^2) = 12 + 4 \w_{q-1}(4)$.
\item[\rm (c)]
If $p > 2$, then 
\[ \GN(p^2 q^2) = 4 + (p^2+p+4)/2 \cdot \w_{q-1}(p^2) 
                  + (p+6) \w_{q-1}(p) 
                  + 2 \w_{q+1}(p) + \w_{q+1}(p^2).\] 
\end{items}
\end{theorem}

\begin{proof}
First, all groups of this order are solvable by Burnside's theorem and it
is obvious that there are 4 isomorphism types of nilpotent groups of this
order. Next, we consider Sylow's theorems. Let $m_q$ denote the number of 
Sylow $q$-subgroups in a group of order $p^2 q^2$ and recall that $p < q$. 
Then $m_q \equiv 1 \bmod q$ and $m_q \mid p^2$. Thus $m_q \in \{1, p, p^2\}$. 
If $m_q = p$, then $q \mid (p-1)$ and this is impossible. If $m_q = p^2$,
then $q \mid p^2-1 = (p-1)(p+1)$. Thus either $q \mid p-1$ or $q \mid p+1$.
Again, this is impossible unless $p = 2$ and $q = 3$. Thus if $(p,q) \neq 
(2,3)$, then $m_q = 1$ and $G$ has a normal Sylow $q$-subgroup. 
\medskip

{\em Non-nilpotent groups with normal Sylow $q$-subgroup.}
These groups have the form $N \split U$ with $|N| = q^2$ and $|U| = p^2$.
We consider the arising cases.
\begin{items}
\item[$\bullet$]
$N \cong C_{q^2}$ and $U \cong C_{p^2}$: Then $\Aut(N) \cong C_{q (q-1)}$ and
thus there are at most one subgroup of order $p$ or $p^2$ in $\Aut(N)$. We
use Theorem \ref{Taunt2} to determine that this case adds
\[ c_1 := \w_{q-1}(p^2) + \w_{q-1}(p).\]
\item[$\bullet$]
$N \cong C_{q^2}$ and $U \cong C_p^2$: This case is similar to the first case
and adds
\[ c_2 := \w_{q-1}(p).\]
\item[$\bullet$]
$N \cong C_q^2$: Let $\varphi : U \ra \Aut(N) = GL(2,q)$ and denote $K = 
ker(\varphi)$. If $|K| = p$, then $U/K \cong C_p$. In both cases on the 
isomorphims type of $U$ there is one $\Aut(U)$-orbit of normal subgroups
of order $p$ and this satisfies $\ol{A}_K = \Aut(U/K)$. Hence in this case
it remains to count the number $c_3$ of conjugacy classes of subgroups of 
order $p$ in $\GL(2,q)$, see Theorem \ref{linear2} (a). Thus this adds
\begin{eqnarray*}
c_3 &=& 2 \mbox{ if } p = 2, \\
c_3 &=& \frac{p+3}{2} \w_{q-1}(p) + \w_{q+1}(p) \mbox{ if } p > 2.
\end{eqnarray*}
If $|K| = 1$, then $U/K \cong U$. Clearly, there is one $\Aut(U)$-orbit of
such normal subgroups and it satisfies $\ol{A}_K = \Aut(U/K)$. Hence in 
this case it remains to count the number $c_4$ of conjugacy classes of 
subgroups of order $p^2$ in $GL(2,q)$, see Theorem \ref{linear2} (b). Thus
this adds
\begin{eqnarray*}
c_4 &=& 2 + 3 \w_{q-1}(4) \mbox{ if } p = 2, \\
c_4 &=& \w_{q-1}(p) + (\frac{p(p+1)}{2}+1) \w_{q-1}(p^2) + \w_{q+1}(p^2) 
        \mbox{ if } p > 2.
\end{eqnarray*}
\end{items}
The number of groups of order $p^2 q^2$ can now be read off as $4 + c_1 + 
c_2 + 2 c_3 + c_4$ and this yields the above formulae.
\end{proof}

\section{Groups of order $p^2 qr$}

The groups of order $p^2 qr$ have been considered by Glenn \cite{Gle06} and
Laue \cite{Lau82}. Glenn's work has several problems. There are groups 
missing from the summary tables, there are duplications and some of the
invariants are not correct; this affects in particular the summary Table 3.
Laue \cite[p. 244-62]{Lau82} does not agree with Glenn. We have not 
compared our results with those of Laue.

\begin{theorem} \label{thp2qr}
Let $p,q$ and $r$ be different primes with $q < r$.
\begin{items}
\item[\rm (a)]
There is one special case $\GN(2^2 \cdot 3 \cdot 5 ) = 13$.
\item[\rm (b)]
If $q = 2$, then 
\[ \GN(p^2 q r) =  10 
   + (2r+7) \w_{p-1}(r) 
   + 3 \w_{p+1}(r)
   + 6 \w_{r-1}(p) 
   + 2 \w_{r-1}(p^2).  \]
\item[\rm (c)]
If $q > 2$, then $\GN(p^2 q r)$ is equal to
\begin{eqnarray*}
   &&  2 + (p^2-p) \w_{q-1}(p^2) \w_{r-1}(p^2) \\
   && + (p-1) (\w_{q-1}(p^2) \w_{r-1}(p)
         +\w_{r-1}(p^2) \w_{q-1}(p)+2 \w_{r-1}(p) \w_{q-1}(p)) \\
   && + (q-1) (q+4)/2 \cdot  \w_{p-1}(q) \w_{r-1}(q)  \\
   && + (q-1)/2 \cdot (\w_{p+1}(q) \w_{r-1}(q) + \w_{p-1}(q) 
           + \w_{p-1}(qr) + 2 \w_{r-1}(pq) \w_{p-1}(q)) \\
   &&   + (qr+1)/2 \cdot  \w_{p-1}(qr)  \\
   &&   + (r+5)/2 \cdot  \w_{p-1}(r) (1 + \w_{p-1}(q)) \\
   &&   + \w_{p^2-1}(qr) + 2 \w_{r-1}(pq) + \w_{r-1}(p) \w_{p-1}(q) 
        + \w_{r-1}(p^2q) \\
   &&   + \w_{r-1}(p) \w_{q-1}(p) + 2 \w_{q-1}(p) 
        + 3 \w_{p-1}(q) + 2 \w_{r-1}(p)  \\
   &&   + 2 \w_{r-1}(q) + \w_{r-1}(p^2)
        + \w_{q-1}(p^2) + \w_{p+1}(r) + \w_{p+1}(q).
\end{eqnarray*}
\end{items}
\end{theorem}

\begin{proof}
The exists one non-solvable group of order $p^2 q r$ and this is the 
group $A_5$ of order $60$. Further, there are two nilpotent groups of 
order $p^2qr$. 
In the following we consider the solvable, non-nilpotent groups $G$ of 
order $p^2 qr$. Let $F$ be the Fitting subgroup of $G$. Then $F$ and 
$G/F$ are both non-trivial and $G/F$ acts faithfully on $\ol{F} := F/\phi(F)$ 
so that no non-trivial normal subgroup of $G/F$ stabilizes a series through 
$\ol{F}$. This yields the following cases.
\begin{items}
\item[$\bullet$]
Case $|F| = p$: In this case $\ol{F} = F$ and $Aut(\ol{F}) = C_{p-1}$. Hence
$G/F$ is abelian and has a normal subgroup isomorphic to $C_p$. This is a
non-trivial normal subgroup which stabilizes a series through $\ol{F}$. As
this cannot occur, this case adds $0$.
\item[$\bullet$]
Case $|F| = q$: In this case $\ol{F} = F$ and $Aut(\ol{F}) = C_{q-1}$ and
$|G/F| = p^2 r$. Thus $p^2 r \mid (q-1)$ and this is impossible, since $r
> q$. Thus this case adds $0$.
\item[$\bullet$]
Case $|F| = r$: In this case $\ol{F} = F$ and $Aut(\ol{F}) = C_{r-1}$ and
$|G/F| = p^2 q$. Hence $G = C_r \split C_{p^2 q}$. Since $Aut(C_r)$ is
cyclic, it has at most one subgroup of order $p^2 q$ and this exists if and
only if $p^2 q \mid (r-1)$. By Theorem \ref{Taunt}, this case adds 
\[ c_1 := \w_{r-1}( p^2 q ).\]
\item[$\bullet$] 
Case $|F| = p^2$ and $F \cong C_{p^2}$. Then $G = C_{p^2} \split C_{qr}$ and
$C_{qr}$ acts faithfully on $C_{p^2}$. Note that $Aut(C_{p^2})$ is cyclic
of order $p(p-1)$. Thus this case arises if and only if $qr \mid (p-1)$ and
in this case there is a unique subgroup of $Aut(F)$ of order $qr$. Again
by Theorem \ref{Taunt}, this case adds
\[ c_2 := \w_{p-1}(qr). \]
\item[$\bullet$] 
Case $|F| = p^2$ and $F \cong C_p^2$. Then $G = C_p^2 \split H$ with $H$
of order $qr$ and $H$ embeds into $Aut(C_p^2) = GL(2,p)$. By Theorem 
\ref{Taunt}, the number of groups $G$ corresponds to the number of conjugacy
classes of subgroups of order $qr$ in $\GL(2,p)$. By Theorem \ref{linear2} (c),
this adds
\begin{eqnarray*}
c_3 &:=& \frac{3r+7}{2} \w_{p-1}(r) + 2 \w_{p+1}(r) \mbox{ if } q = 2, \\
c_3 &:=& \frac{qr + q + r + 5}{2} \w_{p-1}(qr) + 
  \w_{p^2-1}(qr)(1-\w_{p-1}(qr)) \mbox{ if } q>2.
\end{eqnarray*}
\item[$\bullet$] 
Case $|F| = pq$: In this case $\phi(F) = 1$ and $Aut(F) = C_{p-1} \times
C_{q-1}$. Thus $G/F$ is abelian and hence $G/F \cong C_p \times C_r$. Note
that $r > q$ and thus $G/F$ acts on $F$ in such a form that $C_p$ acts as
subgroup of $C_{q-1}$ and $C_r$ acts as subgroup of $C_{p-1}$. This implies
that $r \mid (p-1)$ and $p \mid (q-1)$. This is a contradiction to $r > q$
and hence this case adds $0$.
\item[$\bullet$] 
Case $|F| = pr$: In this case $\phi(F) = 1$ and $Aut(F) = C_{p-1} \times
C_{r-1}$. Thus $G/F$ is abelian and hence $G/F \cong C_p \times C_q$. It
follows that $p \mid (r-1)$ and $q \mid (p-1)(r-1)$. There are two 
cases to distinguish. First, suppose that the Sylow $q$-subgroup of $G/F$
acts non-trivially on the Sylow $p$-subgroup of $F$. Then $q \mid (p-1)$ 
and $G$ splits over $F$ by \cite[Th. 14]{DEi05}. Thus the group $G$ has
the form $F \split_\varphi G/F$ for a monomomorphism $\varphi : G/F \ra 
\Aut(F)$. As in Theorem \ref{Taunt}, the number of such groups $G$ is
given by the number of subgroups of order $pq$ in $\Aut(F)$. As the image 
of the Sylow $p$-subgroup of $G/F$ under $\varphi$ is uniquely determined, 
it remains to evaluate the number of subgroups of order $q$ in $\Aut(F)$ 
that act non-trivally on the Sylow $p$-subgroup of $F$. This number is 
$1 + (q-1) \w_{r-1}(q)$. As second case suppose
that $C_q$ acts trivially on $C_p$. Then $q \mid (r-1)$ and the action of 
$G/F$ on $F$ is uniquely determined. It remains to determine the number of
extensions of $G/F$ by $F$. By \cite[Th. 14]{DEi05} there exist two 
isomorphism types of extensions in this case. In summary, this case adds
\[ c_4 := \w_{r-1}(p)
    ( \w_{p-1}(q) (1+ (q-1)\w_{r-1}(q)) +2 \w_{r-1}(q) ). \]
\item[$\bullet$]
Case $|F| = qr$: In this case $G$ has the form $G \cong N \split_\varphi U$
with $N \cong C_q \times C_r$ and $|U| = p^2$ and $\varphi : U \ra \Aut(N)$
a monomorphism. By Theorem \ref{Taunt}, we have to count the number of 
subgroups
of order $p^2$ in $\Aut(N)$. Note that $\Aut(N) \cong C_{q-1} \times C_{r-1}$.
Thus the number of subgroups of the form $C_p^2$ in $\Aut(N)$ is 
\[ c_{5a} := \w_{q-1}(p) \w_{r-1}(p).\]
It remains to consider the number of cyclic subgroups of order $p^2$ in 
$\Aut(N)$. This number depends on $\gcd(q-1, p^2) = p^a$ and 
$\gcd(r-1,p^2) = p^b$. If $a, b \leq 1$, then this case does not arise. Thus 
suppose that $a = 2$. If $b = 0$, then this yields 1 group. If $b = 1$, then 
this yields $p$ groups. If $b = 2$, then this yields $p(p+1)$ groups. A similar 
results holds for the dual case $b = 2$. We obtain that this adds
\begin{eqnarray*}
 c_{5b} &:=&
       \w_{q-1}(p^2)(1-\w_{r-1}(p))                \\  
   &+& p \w_{q-1}(p^2)\w_{r-1}(p)(1-\w_{r-1}(p^2))  \\  
   &+& p(p+1)\w_{q-1}(p^2)\w_{r-1}(p^2)            \\  
   &+& \w_{r-1}(p^2)(1-\w_{q-1}(p))                \\  
   &+& p \w_{r-1}(p^2)\w_{q-1}(p)(1-\w_{q-1}(p^2)).     
\end{eqnarray*}
In summary, this case adds 
\begin{eqnarray*}
c_5 &=& c_{5a} + c_{5b} \\
    &=& (p^2 -p ) \w_{q-1}(p^2) \w_{r-1}(p^2) \\
      &&+ (p-1) (\w_{q-1}(p^2) \w_{r-1}(p) + \w_{r-1}(p^2) \w_{q-1}(p)) \\
      &&+ \w_{q-1}(p^2) + \w_{r-1}(p^2) + \w_{q-1}(p) \w_{r-1}(p). 
\end{eqnarray*}
\item[$\bullet$]
Case $|F| = pqr$: In this case $G$ has the form $G \cong N \split_\varphi U$
with $N \cong C_q \times C_r$ and $|U| = p^2$ and $\varphi$ has a kernel $K$
of order $p$. Again, we use Theorem \ref{Taunt} to count the number of arising
cases. The group $U$ is either cyclic or $U \cong C_p^2$. In both cases,
there is one $\Aut(U)$-class of normal subgroups $K$ of order $p$ in $U$ and 
this satsifies that $\Aut_K(U)$ maps surjectively on $\Aut(U/K)$. Thus
it remains to count the number of subgroups of order $p$ in $\Aut(N)$. As
$\Aut(N) = C_{q-1} \times C_{r-1}$, this case adds
\[ c_6 := 2( \w_{q-1}(p) + \w_{r-1}(p) + (p-1)\w_{q-1}(p)\w_{r-1}(p)).\]
\item[$\bullet$]
Case $|F| = p^2 q$: Then $G \cong F \split_\varphi U$ with $F = N \times C_q$ 
and $|N| = p^2$ and $U \cong C_r$. By Theorem \ref{Taunt2} we have to count
the number of conjugacy class representatives of subgroups of order $r$ in 
$\Aut(F)$. Recall that $r > q$ and $\Aut(F) = \Aut(N) \times C_{q-1}$. Thus
it remains to count the number of conjugacy classes of subgroups of order
$r$ in $\Aut(N)$. If $N$ is cyclic, then $\Aut(N) \cong C_{p(p-1)}$ and this 
number is $\w_{p-1}(r)$. If $N \cong C_p^2$, then $\Aut(N) \cong \GL(2,p)$
and this number is determined in Theorem \ref{linear2} (a). In summary, this
case adds
\[ c_7 := \frac{r+5}{2} \w_{p-1}(r) + \w_{p+1}(r).  \]
\item[$\bullet$]
Case $|F| = p^2 r$: This case is dual to the previous case with the exception
that now the bigger prime $r$ is contained in $|F|$. A group of this type  
has the form $N \split_\varphi U$ with $N$ nilpotent of order $p^2r$ and 
$U \cong C_q$. We consider the two cases on $N$. \\
If $N$ is cyclic, then $N = C_{p^2} \times C_r$ and $\Aut(N) = C_{p(p-1)}
\times C_{r-1}$. In this case we can use Theorem \ref{Taunt2} to count the
desired subgroups as
\[ c_{8a} := \w_{p-1}(q) + \w_{r-1}(q) + (q-1) \w_{p-1}(q) \w_{r-1}(q). \]
Now we consider the case that $N = C_p^2 \times C_r$. In this case we use a 
slightly different approach and note that all groups in this case have the 
form $M \split_\varphi V$ with $M = C_p^2$ and $|V| = qr$. If $V$ is cyclic,
then $\varphi : V \ra \Aut(M)$ has a kernel $K$ of order $r$. The other 
option is that $V$ is non-abelian and thus of the form $V = C_r \split C_q$. 
In this case the kernel $K$ of $\varphi$
has order $r$ or $qr$. We use Theorem \ref{Taunt} to count the number of
arising cases. If $V$ is cyclic, then it is sufficient to count the 
conjugacy classes of subgroups of order $q$ in $\Aut(M)$. By Theorem 
\ref{linear2} this adds
\[ c_{8b} := s_q(\GL(2,p)).\]
If $V$ is non-abelian, then $q \mid (r-1)$. We consider Theorem \ref{Taunt}
in more detail. First, suppose that $|K| = qr$. Then $\varphi$ is trivial and 
uniquely defined. Thus it remains to consider the case that $|K| = r$. Then 
the subgroup $A_K \leq \Aut(V/K)$ induced by $Stab_{\Aut(V)}(K)$ is the 
trivial group. Let $\SS$ be a set of conjugacy class representatives of 
subgroups of order $q$ in $\GL(2,p)$ and for $S \in \SS$ let $A_S$ denote
the subgroup of $\Aut(V/K)$ induced by the normalizer of $S$ in $\GL(2,p)$. 
Then Theorem \ref{Taunt} yields that the number of groups arising in this 
case is 
\[ c_{8c} := \w_{r-1}(q)(1 + \sum_{S \in \SS} \Aut(V/K) / A_S.\]
Next, note that $A_S$ can be determined via Remark \ref{norms}. As 
$\Aut(V/K) = C_{q-1}$, it follows that 
\begin{eqnarray*}
c_{8c} &=& 3 \mbox{ if } q = 2, \\
c_{8c} &=& \w_{r-1}(q)
           (1 + \frac{(q-1)(q+2)}{2} \w_{p-1}(q) + \frac{q-1}{2} \w_{p+1}(q))
          \mbox{ otherwise. }
\end{eqnarray*}
In summary, this case adds $c_8 = c_{8a} + c_{8b} + c_{8c}$ and this can
be evaluated to $c_8 = 8$ if $q = 2$ and if $q > 2$ then 
\begin{eqnarray*}
c_8 &=&  \frac{(q-1)(q+4)}{2} \w_{p-1}(q) \w_{r-1}(q) \\
        &&+ \frac{q-1}{2} \w_{p+1}(q) \w_{r-1}(q) \\
        &&+ \frac{q+5}{2} \w_{p-1}(q)  \\
        &&+ 2 \w_{r-1}(q) + \w_{p+1}(q).  
\end{eqnarray*}
\end{items}
It now remains to sum up the values for the different cases to determine
the final result. 
\end{proof}

\section{Final comments}

The enumerations of this paper all translate to group constructions. It
would be interesting to make this more explicit and thus to obtain a 
complete and irredundant list of isomorphism types of groups of all orders
considered here. 

\def\cprime{$'$} \def\cprime{$'$}

\end{document}